\documentclass{amsart}
\usepackage{amssymb}
\newtheorem*{thma}{Theorem~A}
\newtheorem*{thmat}{Theorem~A'}
\newtheorem*{thmb}{Theorem~B}
\newtheorem{thm}{Theorem}[section]
\newtheorem{cor}[thm]{Corollary}
\newtheorem{prop}[thm]{Proposition}
\newtheorem{fact}[thm]{Fact}
\newtheorem{lemma}[thm]{Lemma}
\newtheorem{claim}{Claim}[thm]

\theoremstyle{definition}
\newtheorem{defn}[thm]{Definition}

\newtheorem{conv}[thm]{Convention}
\newtheorem*{q}{Open problem}

\theoremstyle{remark}
\newtheorem{remark}[thm]{Remark}

\newcommand*\axiomfont[1]{\textsf{\textup{#1}}}
\newcommand\ch{\axiomfont{CH}}
\newcommand\zfc{\axiomfont{ZFC}}
\newcommand\s{\subseteq}
\newcommand\br{\blacktriangleright}
\renewcommand{\restriction}{\mathbin\upharpoonright}
\renewcommand\mid{\mathrel{|}\allowbreak}
\newcommand{\stick}{{{\ensuremath \mspace{2mu}\mid\mspace{-12mu} {\raise0.6em\hbox{$\bullet$}}}}}

\DeclareMathOperator{\ult}{Ult}
\DeclareMathOperator{\crit}{crit}
\DeclareMathOperator{\id}{id}
\DeclareMathOperator{\reg}{Reg}
\DeclareMathOperator{\cf}{cf}
\DeclareMathOperator{\Tr}{Tr}
\DeclareMathOperator{\tr}{tr}
\DeclareMathOperator{\im}{Im}
\DeclareMathOperator{\otp}{otp}

\DeclareMathOperator{\acc}{acc}
\DeclareMathOperator{\nacc}{nacc}
\DeclareMathOperator{\U}{U}
\DeclareMathOperator{\pr}{Pr}
\DeclareMathOperator{\ssup}{ssup}

\author{Assaf Rinot}
\address{Department of Mathematics, Bar-Ilan University, Ramat-Gan 5290002, Israel.}
\urladdr{http://www.assafrinot.com}

\author{Jing Zhang}
\address{Department of Mathematics, Bar-Ilan University, Ramat-Gan 5290002, Israel.}
\urladdr{https://jingjzzhang.github.io/}


\title{Complicated colorings, revisited}
\begin{document}
\date{Preprint as of February 16, 2022. For the latest version, visit \textsf{http://p.assafrinot.com/52}.}
\maketitle
\begin{abstract} 
In a paper from 1997, Shelah asked whether $\pr_1(\lambda^+,\lambda^+,\lambda^+,\lambda)$ holds for every inaccessible cardinal $\lambda$.
Here, we prove that an affirmative answer follows from $\square(\lambda^+)$. 
Furthermore, we establish that for every pair $\chi<\kappa$ of regular uncountable cardinals, $\square(\kappa)$ implies $\pr_1(\kappa,\kappa,\kappa,\chi)$.
\end{abstract}

\section{Introduction}

The subject matter of this paper is the following two anti-Ramsey coloring principles:

\begin{defn}[Shelah, \cite{Sh:282}]\label{def_pr1}
$\pr_1(\kappa, \kappa, \theta, \chi)$ asserts the existence of a coloring $c:[\kappa]^2 \rightarrow \theta$ such that for every  $\sigma<\chi$, every  
pairwise disjoint subfamily $\mathcal{A}  \subseteq [\kappa]^{\sigma}$ of size $\kappa$,
and every $\tau< \theta$, there is $(a,b) \in [\mathcal{A}]^2$ such that $c[a \times b] = \{\tau\}$.
\end{defn}

\begin{defn}[Lambie-Hanson and Rinot, \cite{paper34}]
  $\U(\kappa, \mu, \theta, \chi)$ asserts the existence of a coloring $c:[\kappa]^2\rightarrow \theta$ such that for every $\sigma < \chi$, every pairwise disjoint subfamily
  $\mathcal{A} \s [\kappa]^{\sigma}$ of size $\kappa$,
  and every $\tau < \theta$, there exists $\mathcal{B} \in [\mathcal{A}]^\mu$
  such that, for every $(a, b) \in [\mathcal{B}]^2$, 
$\min(c[a\times b])\ge\tau$.\footnote{Note that $\pr_1(\kappa,\kappa,\theta,\chi)$ implies $\U(\kappa,2,\theta,\chi)$. 
However, by \cite[Theorem~3.3]{paper35}, it does not imply $\U(\kappa,\kappa,\theta,\chi)$.}
  \end{defn}
The importance of this line of study --- especially in proving instances of $\pr_1(\ldots)$ and $\U(\ldots)$ with a large value of the $4^{\text{th}}$ parameter --- is explained in details in the introductions to \cite{paper18,paper15,paper34}.
In what follows, we survey a few milestone results, depending on the identity of $\kappa$.

$\br$ At the level of the first uncountable cardinal $\kappa=\aleph_1$, the picture is complete:
In his seminal paper \cite{TodActa}, Todor{\v{c}}evi{\'c  proved that $\pr_1(\aleph_1,\aleph_1,\aleph_1,2)$ holds,
improving upon a classic result of Sierpi{\'n}ski \cite{MR1556708} asserting that $\pr_1(\aleph_1,\aleph_1,2,2)$ holds.
In 1980, Galvin \cite{galvin} proved that $\pr_1(\aleph_1,\aleph_1,\theta,\aleph_0)$ is independent of $\zfc$ for any cardinal $\theta\in[2,\aleph_1]$.
Finally, a few years ago, by pushing further ideas of Moore \cite{lspace}, Peng and Wu \cite{MR3742590} 
proved that $\pr_1(\aleph_1,\aleph_1,\aleph_1,\chi)$ holds for every $\chi\in[2,\aleph_0)$.
As for the other coloring principle, and in contrast with Galvin's result, by \cite{paper34}, $\U(\aleph_1,\aleph_1,\theta,\aleph_0)$ holds for any cardinal $\theta\in[2,\aleph_1]$.

$\br$ At the level of the second uncountable cardinal, $\kappa=\aleph_2$,
a celebrated result of Shelah \cite{Sh:572} asserts that $\pr_1(\aleph_2,\aleph_2,\aleph_2,\aleph_0)$ is a theorem of $\zfc$.
Ever since, the following problem remained open:
\begin{q}[Shelah, \cite{Sh:572,Sh:1027}]\begin{enumerate} 
\item Does $\pr_1(\aleph_2,\aleph_2,\aleph_2,\aleph_1)$ hold? 

\item Does $\pr_1(\lambda^+,\lambda^+,\lambda^+,\lambda)$ hold for $\lambda$ inaccessible?
\end{enumerate}
\end{q}

In comparison, by \cite{paper34}, $\U(\lambda^+,\lambda^+,\theta,\lambda)$ is a theorem of $\zfc$ for every infinite regular cardinal $\lambda$ and every cardinal $\theta\in[2,\lambda^+]$.

$\br$ At the level of $\kappa=\lambda^+$ for $\lambda$ a singular cardinal, the main problem left open has to do with the $3^{\text{rd}}$ parameter of  $\pr_1(\ldots)$ rather than the $4^{\text{th}}$ (see \cite{MR1318912,EiSh:535,EiSh:819,MR2652193,MR3087058,MR3087059}).
This is a consequence of three findings. 
First, by the main result of \cite{paper13}, for every singular cardinal $\lambda$ and every cardinal $\theta\le\lambda^+$,
$\pr_1(\lambda^+,\lambda^+,\theta,2)$ implies $\pr_1(\lambda^+,\lambda^+,\theta,\cf(\lambda))$.
Second, by \cite[\S2]{paper45}, if $\lambda$ is the singular limit of strongly compact cardinals,
then $\pr_1(\lambda^+,\lambda^+,2,(\cf(\lambda))^+)$ fails, meaning the the first result cannot be improved. Third,
by \cite[\S2]{paper45}, $\pr_1(\lambda^+,\lambda^+,2,\lambda)$ outright fails for every singular cardinal $\lambda$.

The situation with $\U(\ldots)$ is slightly better. An analog of the first result may be found as \cite[Lemma~2.5 and Theorem~4.21(3)]{paper34}.
An analog of the second result may be found as \cite[Theorem~2.14]{paper34}.
In contrast, by \cite[Corollary~4.15]{paper34}, it is in fact consistent that $\U(\lambda^+,\lambda^+,\theta,\lambda)$ holds for every singular cardinal $\lambda$ and every cardinal $\theta\in[2,\lambda^+]$.

$\br$ At the level of a Mahlo cardinal $\kappa$, 
by \cite[Conclusion~4.8(2)]{Sh:365}, 
the existence of a stationary subset of $E^\kappa_{\ge\chi}$ that does not reflect at inaccessibles entails that
 $\pr_1(\kappa,\kappa,\theta,\chi)$ holds for all $\theta<\kappa$.
By \cite[\S5]{paper45}, the existence of nonreflecting a stationary subset of $\reg(\kappa)$ on which $\diamondsuit$ holds
entails that  $\pr_1(\kappa,\kappa,\kappa,\kappa)$ holds.

The situation with $\U(\ldots)$ is analogous:
By \cite[Theorem~4.23]{paper34}, the existence of a stationary subset of $E^\kappa_{\ge\chi}$ that does not reflect at inaccessibles entails that
$\U(\kappa,\kappa,\theta,\chi)$ holds for all $\theta<\kappa$.
By \cite[\S2]{paper36}, the existence of nonreflecting a stationary subset of $\reg(\kappa)$
entails that  $\U(\kappa,\kappa,\theta,\kappa)$ holds for all $\theta\le\kappa$.

$\br$ At the level of an abstract regular cardinal $\kappa\ge\aleph_2$, we mention two key results.
First, by \cite{paper15}, for every regular cardinal $\kappa\ge\aleph_2$
and every $\chi\in\reg(\kappa)$ such that $\chi^+<\kappa$, the existence of a nonreflecting stationary subset of $E^\kappa_{\ge\chi}$ entails that $\pr_1(\kappa,\kappa,\kappa,\chi)$ holds
(this is optimal, by \cite[Theorem~3.4]{paper35}, it is consistent that for some inaccessible cardinal $\kappa$,
$E^\kappa_\chi$ admits a nonreflecting stationary set, and yet, $\pr_1(\kappa,\kappa,\kappa,\chi^+)$ fails).
Second, by \cite{paper18}, for every regular cardinal $\kappa\ge\aleph_2$
and every $\chi\in\reg(\kappa)$ such that $\chi^+<\kappa$, $\square(\kappa)$ entails that $\pr_1(\kappa,\kappa,\kappa,\chi)$ holds.

Here, the situation with $\U(\ldots)$ is again better. By \cite[Corollaries 4.12 and 4.15]{paper34} and \cite[\S4]{paper36}, the analogs of the two results are true even without requiring ``$\chi^+<\kappa$''!

\vspace{10pt}

After many years without progress on the above mentioned Open Problem, in the last few years, there have been a few breakthroughs.
In an unpublished note from 2017, Todor{\v{c}}evi{\'c proved that $\ch$ implies a weak form of $\pr_1(\aleph_2,\aleph_2,\aleph_2,\aleph_1)$,
strong enough to entail one of its intended applications (the existence of a $\sigma$-complete $\aleph_2$-cc partial order whose square does not satisfy that $\aleph_2$-cc).
Next, in \cite[\S6]{paper45}, the authors obtained a full lifting of Galvin's strong coloring theorem, proving that for every infinite regular cardinal $\lambda$,
$\pr_1(\lambda^+,\lambda^+,\lambda^+,\lambda)$ holds assuming the stick principle $\stick(\lambda^+)$.
In particular, an affirmative answer to (1) follows from $2^{\aleph_1}=\aleph_2$.
Then, very recently, in \cite{MR4323594}, Shelah proved that
for every regular uncountable cardinal $\lambda$,
$\pr_1(\lambda^+,\lambda^+,\lambda^+,\lambda)$ holds assuming the existence of a nonreflecting stationary subset of $E^{\lambda^+}_{<\lambda}$.
So, by a standard fact from inner model theory, a negative answer to (1) implies that $\aleph_2$ is a Mahlo cardinal in G\"odel's constructible universe.

The main result of this paper reads as follows:
\begin{thma} For every regular uncountable cardinal $\lambda$, if $\square(\lambda^+)$ holds, then so does $\pr_1(\lambda^+,\lambda^+,\lambda^+,\lambda)$.
In particular, a negative answer to (1) implies that $\aleph_2$ is a weakly compact cardinal in G\"odel's constructible universe.
\end{thma}

Thanks to the preceding theorem, we can now waive the hypothesis ``$\chi^+<\kappa$'' from \cite[Theorem~B]{paper18},
altogether getting a clear picture:
\begin{thmat}  For every pair $\chi<\kappa$ of regular uncountable cardinals, $\square(\kappa)$ implies $\pr_1(\kappa,\kappa,\kappa,\chi)$.
\end{thmat}

Now, let us say a few words about the proof.
As made clear by the earlier discussion, in the case that $\kappa=\chi^+$, it is easier to prove $\U(\kappa,\kappa,\theta,\chi)$ than proving $\pr_1(\kappa,\kappa,\theta,\chi)$.
Therefore, we consider the following slight strengthening of $\U(\dots)$:

\begin{defn}
  $\U_1(\kappa, \mu, \theta, \chi)$ asserts the existence of a coloring $c:[\kappa]^2\rightarrow \theta$ such that for every $\sigma < \chi$, every pairwise disjoint subfamily
  $\mathcal{A} \s [\kappa]^{\sigma}$ of size $\kappa$,
  and every $\epsilon < \theta$, there exists $\mathcal{B} \in [\mathcal{A}]^\mu$
  such that, for every $(a, b) \in [\mathcal{B}]^2$, 
  there exists $\tau>\epsilon$ such that $c[a\times b]=\{\tau\}$.
\end{defn}

Shelah's proof from \cite{MR4323594} can be described as utilizing the hypothesis of his theorem twice: 
first to get $\U_1(\lambda^+,2,\lambda^+,\lambda)$,
and then to derive $\pr_1(\lambda^+,\lambda^+,\lambda^+,\lambda)$ from the latter.
Here, we shall follow a similar path,
building on the progress made in \cite[\S5]{paper44} with respect to walking along well-chosen $\square(\kappa)$-sequences.
We shall also present a couple of propositions translating $\U_1(\ldots)$ to $\pr_1(\ldots)$ and vice versa, 
demonstrating that $\U_1(\kappa,\mu,\theta,\chi)$ is of interest also with $\theta<\kappa$.
For instance, it will be proved that for every regular uncountable cardinal $\lambda$ that admits a stationary set not reflecting at inaccessibles (e.g., $\lambda=\aleph_1$),
 $\U_1(\lambda^+,2,\lambda,\lambda)$ iff $\pr_1(\lambda^+,\lambda^+,\lambda^+,\lambda)$.
Thus,  the core contribution of this paper reads as follows.
\begin{thmb} 
Suppose that $\chi\le\theta\le\kappa$ are infinite regular cardinals such that $\max\{\chi,\aleph_1\}<\kappa$.
If $\square(\kappa)$ holds, then so does $\U_1(\kappa,2,\theta,\chi)$.
\end{thmb}

\section{Preliminaries}
In what follows, $\chi<\kappa$ denotes a pair of infinite regular cardinals.
$\reg(\kappa)$ stands for the set of all infinite and regular cardinals below $\kappa$.
Let $E^\kappa_\chi:=\{\alpha < \kappa \mid \cf(\alpha) = \chi\}$,
and define $E^\kappa_{\le \chi}$, $E^\kappa_{<\chi}$, $E^\kappa_{\ge \chi}$, $E^\kappa_{>\chi}$,  $E^\kappa_{\neq\chi}$ analogously.
A stationary subset $S\s\kappa$ is \emph{nonreflecting} (resp.~\emph{nonreflecting at inaccessibles}) iff there exists no $\alpha\in E^\kappa_{>\omega}$ (resp.~$\alpha$ a regular limit uncountable cardinal) such that $S\cap\alpha$ is stationary in $\alpha$.
For a set of ordinals $a$, we write 
$\ssup(a):=\sup\{\alpha+1\mid \alpha\in a\}$,
$\acc^+(a) := \{\alpha < \ssup(a) \mid \sup(a \cap \alpha) = \alpha > 0\}$,
$\acc(a) := a \cap \acc^+(a)$ and $\nacc(a) := a \setminus \acc(a)$.
For sets of ordinals that are not ordinals, $a$ and $b$, we write $a < b$ to express that $\alpha<\beta$ for all $\alpha\in a$ and $\beta\in b$.
For an ordinal $\sigma$ and a set of ordinals $A$, we write 
$[A]^\sigma$ for $\{ B\s A\mid \otp(B)=\sigma\}$.
In the special case that $\sigma=2$ and $\mathcal{A}$ is either an ordinal or a collection of sets
of ordinals, we interpret $[\mathcal{A}]^2$ as the collection of \emph{ordered} pairs $\{ (a,b)\in\mathcal A\times\mathcal A\mid a<b\}$.
In particular, $[\kappa]^2=\{(\alpha,\beta)\mid \alpha<\beta<\kappa\}$.

For the rest of this section, let us fix a \emph{$C$-sequence} $\vec C=\langle C_\alpha\mid\alpha<\kappa\rangle$ over $\kappa$, i.e.,
for every $\alpha<\kappa$, $C_\alpha$ is a closed subset of $\alpha$ with $\sup(C_\alpha)=\sup(\alpha)$.
The next definition is due to Todor{\v{c}}evi{\'c; see \cite{TodWalks} for a comprehensive treatment.
\begin{defn}[Todor{\v{c}}evi{\'c}] From $\vec C$, derive maps $\Tr:[\kappa]^2\rightarrow{}^\omega\kappa$,
$\rho_2:[\kappa]^2\rightarrow	\omega$,
$\tr:[\kappa]^2\rightarrow{}^{<\omega}\kappa$ and $\lambda:[\kappa]^2\rightarrow\kappa$, as follows.
Let $(\alpha,\beta)\in[\kappa]^2$ be arbitrary.
\begin{itemize}
\item $\Tr(\alpha,\beta):\omega\rightarrow\kappa$ is defined by recursion on $n<\omega$:
$$\Tr(\alpha,\beta)(n):=\begin{cases}
\beta,&n=0\\
\min(C_{\Tr(\alpha,\beta)(n-1)}\setminus\alpha),&n>0\ \&\ \Tr(\alpha,\beta)(n-1)>\alpha\\
\alpha,&\text{otherwise}
\end{cases}$$
\item $\rho_2(\alpha,\beta):=\min\{n<\omega\mid \Tr(\alpha,\beta)(n)=\alpha\}$;
\item $\tr(\alpha,\beta):=\Tr(\alpha,\beta)\restriction \rho_2(\alpha,\beta)$;
\item $\lambda(\alpha,\beta):=\max\{ \sup(C_{\Tr(\alpha,\beta)(i)}\cap\alpha) \mid i<\rho_2(\alpha,\beta)\}$.
\end{itemize}
\end{defn}

\begin{conv} From any coloring $h:\kappa\rightarrow\kappa$, derive a function $\tr_h:[\kappa]^2\rightarrow{}^{<\omega}\kappa$ via
$$\tr_h(\alpha,\beta):=\langle h(\Tr(\alpha,\beta)(i))\mid i<\rho_2(\alpha,\beta)\rangle.$$
\end{conv}

The next fact is quite elementary. See, e.g., \cite[Claim~3.1.2]{paper15} for a proof.
\begin{fact}\label{fact2} Whenever $\lambda(\gamma,\beta)<\alpha<\gamma<\beta<\kappa$, 
$\tr(\alpha,\beta)=\tr(\gamma,\beta){}^\smallfrown \tr(\alpha,\gamma)$.
\end{fact}

We now recall the characteristic $\lambda_2(\cdot,\cdot)$, a variation of $\lambda(\cdot,\cdot)$ having the property 
that $\lambda_2(\gamma, \beta) < \gamma$ whenever $0<\gamma < \beta < \kappa$.

\begin{defn}[\cite{paper18}] Define $\lambda_2:[\kappa]^2\rightarrow\kappa$ via
$$\lambda_2(\alpha,\beta):=\sup(\alpha\cap\{ \sup(C_\delta\cap\alpha) \mid \delta \in \im(\tr(\alpha, \beta))\}).$$
\end{defn}

\begin{fact}[{\cite[Lemma~4.7]{paper34}}]\label{lambda2}
Suppose that $\lambda_2(\gamma,\beta)<\alpha<\gamma<\beta<\kappa$.
 
Then $\tr(\alpha,\beta)$ end-extends $\tr(\gamma,\beta)$, and one of the following cases holds:
\begin{enumerate}
\item $\gamma\in\im(\tr(\alpha,\beta))$; or
\item $\gamma\in\acc(C_\delta)$ for $\delta:=\min(\im(\tr(\gamma,\beta)))$.
\end{enumerate}
\end{fact}

\begin{conv}[\cite{paper44}]\label{etanotation}
For every ordinal $\eta<\kappa$ and a pair $(\alpha,\beta)\in[\kappa]^2$, let
$$\eta_{\alpha,\beta}:=\min\{ n<\omega\mid \eta\in C_{\Tr(\alpha,\beta)(n)}\text{ or }n=\rho_2(\alpha,\beta)\}+1.$$
\end{conv}

\begin{defn}[{\cite[\S3]{paper45}}] $\chi_1(\vec C)$ stands for the supremum of $\sigma+1$ over all $\sigma<\kappa$ satisfying the following.
For every pairwise disjoint subfamily $\mathcal A\s[\kappa]^{\sigma}$ of size $\kappa$, 
there are a stationary set $\Delta\s\kappa$ and an ordinal $\eta<\kappa$
such that, for every $\delta\in\Delta$, there exist $\kappa$ many $b\in\mathcal A$ such that,
for every $\beta\in b$,
$\lambda(\delta,\beta)=\eta$ and $\rho_2(\delta,\beta)=\eta_{\delta,\beta}$.
\end{defn}

\begin{fact}[{\cite[\S3]{paper45}}]\label{chi1fact} 
If the two hold:
\begin{enumerate}
\item[($\aleph$)] for all $\alpha<\kappa$ and $\delta\in\acc(C_\alpha)$, $C_{\delta}=C_\alpha\cap\delta$;
\item[($\beth$)] for every club $D\s\kappa$, there exists $\gamma>0$ with $\sup(\nacc(C_\gamma)\cap D)=\gamma$,
\end{enumerate}
then $\chi_1(\vec C)=\sup(\reg(\kappa))$.
\end{fact}

\begin{defn}[Todor{\v{c}}evi{\'c}, \cite{TodActa}]\label{sqk} For a cardinal $\mu\le\kappa$,
$\square(\kappa,{<}\mu)$ asserts the existence of a sequence $\vec{\mathcal C}=\langle\mathcal C_\alpha\mid\alpha<\kappa\rangle$ such that
\begin{enumerate}
\item for every $\alpha<\kappa$, $\mathcal C_\alpha$ is nonempty collection of less than $\mu$ many closed subsets $C$ of $\alpha$ with $\sup(C)=\sup(\alpha)$;
\item for all $\alpha<\kappa$, $C\in\mathcal C_\alpha$ and $\delta\in\acc(C)$, $C\cap\delta\in\mathcal C_\delta$;
\item there exists no club $C$ in $\kappa$ such that $C\cap\alpha\in\mathcal C_\alpha$ for all $\alpha\in\acc(C)$.
\end{enumerate}
\end{defn}

The special case of $\square(\kappa,{<}\mu)$ with $\mu=2$ is denoted by $\square(\kappa)$.

\begin{fact}[Hayut and Lambie-Hanson, {\cite[Lemma~2.4]{MR3730566}}]\label{weakthread}  Clause~(3) of Definition~\ref{sqk} is preserved in any $\kappa$-cc forcing extension,
provided that $\mu<\kappa$.
\end{fact}

\section{Theorem~B}
\begin{thm}\label{thmb} Suppose that $\chi\le\theta\le\kappa$ are infinite regular cardinals such that
$\max\{\chi,\aleph_1\}<\kappa$.
If $\square(\kappa)$ holds, then so does $\U_1(\kappa,2,\theta,\chi)$.
\end{thm}
\begin{proof} Suppose that $\square(\kappa)$ holds.
Then, by \cite[Lemma~5.1]{paper44}, we may fix a $C$-sequence $\vec C=\langle C_\alpha\mid\alpha<\kappa\rangle$ satisfying the following:
\begin{enumerate}
\item $C_{\alpha+1}=\{0,\alpha\}$ for every $\alpha<\kappa$;
\item for every club $D\s\kappa$, there exists $\gamma>0$ with $\sup(\nacc(C_\gamma)\cap D)=\gamma$;
\item for every $\alpha\in\acc(\kappa)$ and $\bar\alpha\in\acc(C_\alpha)$, $C_{\bar\alpha}=C_\alpha\cap\bar\alpha$;
\item for every $i<\kappa$, $\{ \alpha<\kappa\mid \min(C_\alpha)=i\}$ is stationary.
\end{enumerate}

Note that, by Fact~\ref{chi1fact}, $\chi_1(\vec C)=\sup(\reg(\kappa))$.
If $\theta<\kappa$, then let $\mu:=\theta$; otherwise, let $\mu:=\chi$.
Derive a coloring $h:\kappa\rightarrow\mu$ via
$$h(\alpha):=\begin{cases}\min(C_\alpha),&\text{if }\min(C_\alpha)<\mu;\\
0,&\text{otherwise}.\end{cases}$$
  
We shall walk along $\vec C$.
Define a coloring $d:[\kappa]^2\rightarrow\mu$ via $$d(\alpha,\beta):=\max(\im(\tr_h(\alpha,\beta))).$$

\begin{claim}\label{claim251} Suppose that $\alpha,\beta,\gamma$ are ordinals, and $\lambda_2(\gamma,\beta)<\alpha<\gamma<\beta<\kappa$.

Then $\im(\tr_h(\alpha,\beta))=\im(\tr_h(\alpha,\gamma))\cup \im(\tr_h(\gamma,\beta))$.
In particular, $d(\alpha,\beta)=\max\{d(\alpha,\gamma),d(\gamma,\beta)\}$.
\end{claim}
\begin{proof}By Fact~\ref{lambda2}, one of the following cases holds:
\begin{enumerate}
\item[$\br$] $\gamma\in\im(\tr(\alpha,\beta))$. 
In this case, $\tr(\alpha,\beta)=\tr(\gamma,\beta){}^\smallfrown\tr(\alpha,\gamma)$,
so we done.
\item[$\br$] $\gamma\in\acc(C_\delta)$ for $\delta:=\min(\im(\tr(\gamma,\beta)))$.
In this case, $\tr(\alpha,\beta)=\tr(\delta,\beta){}^\smallfrown\allowbreak\tr(\alpha,\delta)$,
so that $\im(\tr_h(\alpha,\beta))=\im(\tr_h(\alpha,\delta))\cup \im(\tr_h(\delta,\beta))$.
Since $\gamma\in\acc(C_\delta)$, Clause~(3) above
and the definition of the function $h$ together imply that
$\tr_h(\alpha,\delta)=\tr_h(\alpha,\gamma)$.
In addition, $\tr(\gamma,\beta)=\tr(\delta,\beta){}^\smallfrown\langle\delta\rangle$,
so that $\im(\tr_h(\gamma,\beta))=\im(\tr_h(\delta,\beta))\cup\{h(\delta)\}$.
Since $h(\delta)\in\im(\tr_h(\alpha,\delta))$,
altogether,    
\[\im(\tr_h(\alpha,\gamma))\cup \im(\tr_h(\gamma,\beta))=\im(\tr_h(\alpha,\delta))\cup \im(\tr_h(\delta,\beta)).\qedhere\]
\end{enumerate}
\end{proof}

We are now ready to define the sought coloring $c$.
If $\mu=\theta$, then let $c:=d$, and otherwise define $c:[\kappa]^2\rightarrow\theta$ via
$$c(\alpha,\beta):=\max\{\xi\in\im(\tr(\alpha,\beta))\mid h(\xi)=d(\alpha,\beta)\}.$$
To see that $c$ witnesses $\U_1(\kappa,2,\theta,\chi)$, 
suppose that we are given $\epsilon<\theta$, $\sigma<\chi$ and a $\kappa$-sized pairwise disjoint subfamily $\mathcal A\s[\kappa]^{\sigma}$;
we need to find $\tau>\epsilon$ and $(a,b)\in[\mathcal A]^2$ such that $c[a\times b]=\{\tau\}$.
As $\sigma<\chi\le\chi_1(\vec C)$, we may fix a stationary subset $\Delta\s\kappa$ and an ordinal $\eta<\kappa$
such that, for every $\delta\in\Delta$, there exists $b\in\mathcal A$ with $\min(b)>\delta$ such that
$\lambda(\delta,\beta)=\eta$ for every $\beta\in b$.
Set $\eta':=\max\{\eta,\epsilon\}$.

Consider the club $C:=\{\gamma<\kappa\mid \sup\{\min(a)\mid a\in\mathcal A\cap\mathcal P(\gamma)\}=\gamma\}$.
For all $\gamma\in C$ and $\varepsilon<\gamma$,
fix $a^\gamma_\varepsilon\in\mathcal A\cap\mathcal P(\gamma)$ with $\min(a^\gamma_\varepsilon)>\varepsilon$;
as $|a^\gamma_\varepsilon|<\mu$,
$\tau^\gamma_\varepsilon:=\sup\{ d(\alpha,\gamma)\mid \alpha\in a^\gamma_\varepsilon\}$ is $<\mu$.
Fix some stationary $\Gamma\s C\cap E^\kappa_{\neq\mu}$ along with $\tau_0<\mu$ such that,
for every $\gamma\in\Gamma$, $\sup\{ \varepsilon<\gamma\mid \tau^\gamma_\varepsilon\le\tau_0\}=\gamma$.

By Clause~(4), for each $i<\mu$, $H_i:=\{\alpha<\kappa\mid h(\alpha)=i\}$ is stationary,
so, fix $\delta\in\Delta\cap\bigcap_{i<\mu}\acc^+(H_i\cap\acc^+(\Gamma\setminus \eta'))$.
Pick $b\in\mathcal A$ with $\min(b)>\delta$ such that
$\lambda(\delta,\beta)=\eta$ for every $\beta\in b$. 
As $|b|<\mu$, $\tau_1:=\sup\{ d(\delta,\beta)\mid \beta\in b\}$ is $<\mu$.
If $\epsilon<\mu$, then pick $\zeta\in H_{\tau_0+\tau_1+\epsilon+1}\cap\acc^+(\Gamma\setminus\eta')$;
otherwise, pick $\zeta\in H_{\tau_0+\tau_1+1}\cap\acc^+(\Gamma\setminus\eta')$.
Next, pick $\gamma\in\Gamma$ above $\max\{\lambda_2(\zeta,\delta),\eta'\}$.
Finally, pick $\varepsilon<\gamma$ above $\max\{\lambda_2(\gamma,\zeta),\lambda_2(\zeta,\delta),\eta'\}$ such that $\tau^\gamma_{\varepsilon}\le\tau_0$,
and then set $a:=a^\gamma_\varepsilon$.

\begin{claim} Let $\alpha\in a$ and $\beta\in b$. Then:
\begin{enumerate}
\item[(i)] $\max\{\lambda_2(\gamma,\zeta),\lambda_2(\zeta,\delta),\lambda(\delta,\beta),\epsilon\}<\varepsilon<\alpha<\gamma<\zeta<\delta<\beta$;
\item[(ii)] $c(\alpha,\beta)=c(\gamma,\delta)>\epsilon$.
\end{enumerate}
\end{claim}
\begin{proof} (i) This is clear, recalling that $\eta'=\max\{\lambda(\delta,\beta),\epsilon\}$.

(ii) From $\lambda(\delta,\beta)<\alpha<\delta<\beta$ and Fact~\ref{fact2}, we infer that $\tr(\alpha,\beta)=\tr(\delta,\beta){}^\smallfrown\allowbreak\tr(\alpha,\delta)$,
so that $d(\alpha,\beta)=\max\{d(\delta,\beta),d(\alpha,\delta)\}$.
By Clause~(i) and Claim~\ref{claim251}, 
$$d(\alpha,\delta)=\max\{d(\alpha,\zeta),d(\zeta,\delta)\}\ge h(\zeta)>\tau_1\ge d(\delta,\beta).$$
Consequently, $d(\alpha,\beta)=d(\alpha,\delta)$ and $c(\alpha,\beta)=c(\alpha,\delta)$.
By Clause~(i) and Claim~\ref{claim251}, $\im(\tr_h(\alpha,\delta))=\im(\tr_h(\alpha,\gamma)\cup\im(\tr_h(\gamma,\delta))$.
As $d(\alpha,\gamma)\le\tau_0<h(\zeta)\le d(\gamma,\delta)$,
it follows that $d(\alpha,\delta)=d(\gamma,\delta)$ and $c(\alpha,\delta)=d(\gamma,\delta)$.
Altogether, $c(\alpha,\beta)=c(\gamma,\delta)$.

Now, if $\theta<\kappa$, then $\epsilon<\theta=\mu$ and $c=d$, so that $c(\alpha,\beta)=d(\gamma,\delta)\ge h(\zeta)>\epsilon$.
Otherwise, $c(\alpha,\beta)\ge\min(\im(\tr(\alpha,\beta)))>\alpha>\epsilon$.
\end{proof}

Set $\tau:=c(\gamma,\delta)$. Then $\tau>\epsilon$ and $c[a\times b]=\{\tau\}$, as sought.
\end{proof}
\begin{remark} The preceding proof makes it clear that the auxiliary coloring $d$ witnesses $\U_1(\kappa,2,\mu,\chi)$.
By Fact~\ref{lambda2}, the coloring $d$ is moreover \emph{closed} 
in the sense that, for all $\beta<\kappa$ and $i<\theta$,
the set $\{\alpha<\beta\mid c(\alpha,\beta)\le i\}$ is closed below $\beta$.
So, by \cite[Lemma~4.2]{paper34}, $d$ witnesses $\U(\kappa,\kappa,\mu,\chi)$, as well.
\end{remark}

\section{Connecting $\U_1$ with $\pr_1$}
Throughout this section, $\chi<\kappa$ is a pair of infinite regular cardinals, and $\theta$ is a regular cardinal $\le\kappa$.
Let $\mathbb A^\kappa_\chi$ denote the collection of all pairwise disjoint subfamilies $\mathcal A\s\mathcal P(\kappa)$ such that $|\mathcal A|=\kappa$ 
and $\sup\{|a|\mid a\in\mathcal A\}<\chi$.
Given a coloring $c:[\kappa]^2\rightarrow\theta$,
for every $\mathcal A\s\mathcal P(\kappa)$, let $T_c(\mathcal A)$ be the set of all $\tau<\theta$ such that, for some $(a,b)\in[\mathcal A]^2$,
$c[a\times b]=\{\tau\}$.
The next definition appears (with a slightly different notation) in Stage~B in the proof of \cite[Theorem~1.1]{MR4323594}:
\begin{defn} For every coloring $c:[\kappa]^2\rightarrow\theta$, let
$$F_{c,\chi}:=\{ T\s\theta\mid \exists \mathcal A\in\mathbb A^\kappa_\chi\,[T_c(\mathcal A)\s T]\}.$$
\end{defn}

\begin{prop}\label{lemma22} Suppose that a coloring $c:[\kappa]^2\rightarrow\theta$ witnesses $\U_1(\kappa,2,\theta,\chi)$,
and $\lambda$ is some cardinal. Then:
\begin{enumerate}
\item  $F_{c,\chi}$ is a $\chi$-complete uniform filter on $\theta$;
\item If every $\chi$-complete uniform filter on $\theta$ is not weakly $\lambda$-saturated, then $\pr_1(\kappa,\kappa,\lambda,\chi)$ holds.
\end{enumerate}
\end{prop}
\begin{proof} (1) It is clear that $F_{c,\chi}$ is upward-closed. To see that it is $\chi$-complete,
suppose that we are given a sequence $\langle X_i\mid i<\delta\rangle$ of elements of $F_{c,\chi}$,
for some $\delta<\chi$.
For each $i<\delta$, fix $\mathcal A_i\in \mathbb A^\kappa_\chi$ such that $T_c(\mathcal A_i)\s X_i$.
Pick  $\mathcal A\in\mathbb A^\kappa_\chi$ such that, for every $a\in\mathcal A$, 
there is a sequence $\langle a_i\mid i<\delta\rangle\in\prod_{i<\delta}\mathcal A_i$ such that $a=\bigcup_{i<\delta}a_i$.
Then, $T_c(\mathcal A)\s\bigcap_{i<\delta}T_c(\mathcal A_i)\s\bigcap_{i<\delta}X_i$ and hence the latter is in $F_{c,\chi}$.
Finally, since $c$ witnesses $\U_1(\kappa,2,\theta,\chi)$, for every $\mathcal A\in\mathbb A^\kappa_\chi$ and every $\epsilon<\theta$,
$T_c(\mathcal A)\setminus\epsilon$ is nonempty. So $F_{c,\chi}$ consists of cofinal subset of $\theta$. Since $\theta$ is regular,
$F_{c,\chi}$ is uniform.

(2) Suppose that no $\chi$-complete uniform filter on $\theta$ is weakly $\lambda$-saturated.
In particular, by Clause~(1), 
we may pick a map $\psi:\theta\rightarrow\lambda$ such that that the preimage of any singleton is $F_{c,\theta}$-positive.
Then $\psi\circ c$ witnesses $\pr_1(\kappa,\kappa,\lambda,\chi)$.
\end{proof}

\begin{cor}\label{cor43} Suppose that $\lambda$ is a regular uncountable cardinal. 

If $\lambda$ admits a stationary set that does not reflect at regulars or if $\square(\lambda,{<}\mu)$ holds for some cardinal $\mu<\lambda$,
then the following are equivalent:
\begin{enumerate}
\item $\pr_1(\lambda^+,\lambda^+,\lambda^+,\lambda)$;
\item $\pr_1(\lambda^+,\lambda^+,\lambda,\lambda)$;
\item  $\U_1(\lambda^+,2,\lambda,\lambda)$.
\end{enumerate}
\end{cor}
\begin{proof} 
The implication $(1)\implies(2)\implies(3)$ is trivial,
and the fact that $(2)\implies(1)$ is well-known (see, for instance, \cite[\S6]{paper49}).
By the preceding proposition, to see that $(3)\implies(2)$, it suffices to prove that under our hypothesis on $\lambda$,
no $\lambda$-complete uniform filter on $\lambda$ is weakly $\lambda$-saturated.
Now, if $\lambda$ is a successor cardinal, then this follows from Ulam's theorem \cite{ulam1930masstheorie}, 
and if $\lambda$ is an inaccessible cardinal admitting a stationary set that does not reflect at regulars, then this follows from a theorem of Hajnal \cite{MR260597}. 
Finally, if $\square(\lambda,{<}\mu)$ holds for some cardinal $\mu<\lambda$,
then this follows from \cite[Theorem~A]{paper53}.
\end{proof}

\begin{lemma}\label{lemma44} Suppose that $\lambda$ is a regular uncountable cardinal and $\square(\lambda^+,{<}\lambda)$ holds.
Then every $\lambda$-complete uniform filter on $\lambda^+$ is not weakly $\lambda$-saturated.
\end{lemma}
\begin{proof} Fix a $\square(\lambda^+,{<}\lambda)$-sequence $\vec{\mathcal C}=\langle \mathcal C_\alpha\mid\alpha<\lambda^+\rangle$.
For each $\alpha<\lambda^+$, fix an injective enumeration $\langle C_{\alpha,i}\mid  i<|\mathcal C_\alpha|\rangle$ of $\mathcal{C}_\alpha$.

Towards a contradiction, suppose that $F$ is a $\lambda$-complete uniform filter on $\lambda^+$ that is weakly $\lambda$-saturated.
Since $F$ is $\lambda$-complete, $F$ is moreover $\lambda$-saturated. Hence, $\mathcal P(\lambda^+)/F$ is a $\lambda$-cc notion of forcing.

Let $G$ be $\mathcal P(\lambda^+)/F$-generic over $V$. Then $G$ is a uniform $V$-ultrafilter over $\lambda^+$ extending $F$.
By \cite[Propositions 2.9 and 2.14]{MR2768692}, $\ult(V,G)$ is well-founded and $j:V\rightarrow M\simeq\ult(V,G)$ satisfies $\crit(j)=\lambda$.

Now, work in $V[G]$. Denote $j(\vec{\mathcal C})$ by $\langle \mathcal D_\alpha\mid \alpha<j(\lambda^+)\rangle$.
For every $\alpha<\lambda^+$, since $\crit(j)=\lambda>|\mathcal C_\alpha|$,  it is the case that $\mathcal D_{j(\alpha)}=j(\mathcal C_\alpha)=j``\mathcal C_\alpha$.
Since $G$ is uniform, $\gamma:=\sup(j``\lambda^+)$ is $<j(\lambda^+)$, as witnessed by the identity map $\id:\lambda^+\rightarrow\lambda^+$.
As $V[G]$ is a $\lambda$-cc forcing extension of $V$, $\cf^V(\gamma)=\cf^{V[G]}(\gamma)=\lambda^{+}$, so that $\cf^M(\gamma)\ge\lambda^+$.
Pick $D\in\mathcal D_\gamma$.
\begin{claim} $A:=j^{-1}[\acc(D)]$ is a cofinal subset of $\lambda^+$.
\end{claim}
\begin{proof}  Given $\epsilon<\lambda^+$, we recursively define (in $V[G]$) an increasing sequence $\langle \alpha_n\mid n<\omega\rangle$ of ordinals below $\lambda^+$ such that:
\begin{enumerate}
\item $\epsilon=\alpha$, and
\item for all $n<\omega$, $(j(\alpha_n), j(\alpha_{n+1})]\cap D \neq \emptyset$.
\end{enumerate}
Consider $\alpha^*:=\sup_{n<\omega}\alpha_n$. Notice that $\cf^V(\alpha^*)<\lambda$, since if $\cf^V(\alpha^*)\geq \lambda$, 
then by the fact that $V[G]$ is a $\lambda$-cc forcing extension of $V$ we have $\omega=\cf^{V[G]}(\alpha^*)\geq \lambda$ which is impossible. 
As a result, $\sup j``\alpha^* = j(\alpha^*)\in\acc(D)$, which implies that $\alpha^*$ is an element of $A$ above $\epsilon$.
\end{proof}

For each $\alpha\in A$, $D\cap j(\alpha)\in \mathcal{D}_{j(\alpha)}=j``\mathcal{C}_\alpha$,
so we may pick some $i_\alpha<\lambda$ such that $D\cap j(\alpha)=j(C_{\alpha,i_\alpha})$. 
Fix some $i<\lambda$ for which $A':=\{\alpha\in A\mid i_\alpha=i\}$ is cofinal in $\lambda^+$.
For every $(\alpha,\beta)\in[A']^2$, 
$j(C_{\alpha,i})=D\cap j(\alpha)$ and $j(C_{\beta,i})=D\cap j(\beta)$, so, 
by elementarity, $C_{\alpha,i}=C_{\beta,i}\cap\alpha$. 
As $A'$ is cofinal in $\lambda^+$, it follows that $C:=\bigcup\{ C_{\alpha,i}\mid \alpha\in A\}$ is a club in $\lambda^+$.
Evidently, $C\cap\alpha\in\mathcal C_\alpha$ for every $\alpha\in\acc(C)$.
However, $V[G]$ is a $\lambda$-cc forcing extension of $V$, contradicting Fact~\ref{weakthread}.
\end{proof} 

We are now ready to prove Theorem~A:
\begin{cor} Suppose that $\lambda$ is a regular uncountable cardinal,
and $\square(\lambda^+)$ holds. Then $\pr_1(\lambda^+,\lambda^+,\lambda^+,\lambda)$ holds, as well.
\end{cor}
\begin{proof}  By Theorem~\ref{thmb}, using $(\kappa,\theta,\chi):=(\lambda^+,\lambda^+,\lambda)$, $\U_1(\lambda^+,2,\lambda^+,\lambda)$ holds. 
So, by Proposition~\ref{lemma22} (using $\theta:=\lambda^+$) and Lemma~\ref{lemma44}, $\pr_1(\lambda^+,\lambda^+,\lambda,\lambda)$ holds. 
Then, again by \cite[\S6]{paper49},
$\pr_1(\lambda^+,\lambda^+,\lambda^+,\lambda)$ holds, as well.
\end{proof}
 
\section*{Acknowledgments}
The first author is partially supported by the European Research Council (grant agreement ERC-2018-StG 802756) and by the Israel Science Foundation (grant agreement 2066/18).
The second author is supported by the Foreign Postdoctoral Fellowship Program of the Israel Academy of Sciences and Humanities and by the Israel Science Foundation (grant agreement 2066/18).

The main result of this paper was presented by the second author at 
the \emph{Israel Mathematical Union Annual Meeting} special session in set theory and logic in July 2021. He thanks the organizers for the invitation and the participants for their feedback.

\end{document}